\documentclass[graybox]{svmult}

\usepackage{mathptmx}       
\usepackage{helvet}         
\usepackage{courier}        
\usepackage{makeidx}         

\usepackage{multicol}        
\usepackage[bottom]{footmisc}

\usepackage[pdftex]{graphicx}
\usepackage{float}
\usepackage{hyperref}
\usepackage{amssymb}
\usepackage{amsmath}

\usepackage[all]{xy}

\newcommand{\be}{\begin{eqnarray}}
\newcommand{\ee}{\end{eqnarray}}
\newcommand{\bez}{\begin{eqnarray*}}
\newcommand{\eez}{\end{eqnarray*}}

\def\NN{{\mathbb N}}

\def\epi{\twoheadrightarrow}
\def\mono{\rightarrowtail}

\def\arbreA{\vcenter{\xymatrix@R=3pt@C=3pt{
&& \\
&*{}\ar@{-}[ur] \ar@{-}[ul] \ar@{-}[d]     &\\
&&
}}}

\def\petitarbreA{\vcenter{\xymatrix@R=1pt@C=1pt{
&& \\
&*{}\ar@{-}[ur] \ar@{-}[ul] \ar@{-}[d]     &\\
&&
}}}

\def\petitarbreB{\vcenter{\xymatrix@R=.4ex@C=.4ex{
&& \\
&*{}\ar@{-}[ur] \ar@{-}[ul] \ar@{-}[d]     &\\
&&
}}}

\def\arbreBA{\vcenter{\xymatrix@R=2pt@C=2pt{
&&&&\\
&&&*{}\ar@{-}[ul] & \\
&&*{}\ar@{-}[uurr] \ar@{-}[uull] \ar@{-}[d]     &&\\
&&&&
}}}

\def\arbreAB{\vcenter{\xymatrix@R=2pt@C=2pt{
&&&&\\
&*{}\ar@{-}[ur] &&& \\
&&*{}\ar@{-}[uurr] \ar@{-}[uull] \ar@{-}[d]     &&\\
&&&&
}}}

\def\arbreBB{\vcenter{\xymatrix@R=2pt@C=2pt{
&&*{}&&\\
&&&& \\
&&*{}\ar@{-}[uurr] \ar@{-}[uull] \ar@{-}[d] \ar@{-}[uu]     &&\\
&&&&
}}}

\def\arbreABC{\vcenter{\xymatrix@R=1pt@C=1pt{
&&&&&&\\
&*{}\ar@{-}[ur] &&&&& \\
&&*{}\ar@{-}[uurr] &&&&\\
&&&*{}\ar@{-}[uuurrr] \ar@{-}[uuulll] \ar@{-}[d] &&&\\
&&&&&&
}}}

\def\arbreBAC{\vcenter{\xymatrix@R=1pt@C=1pt{
&&&&&&\\
&&&*{}\ar@{-}[ul] &&& \\
&&*{}\ar@{-}[uurr] &&&&\\
&&&*{}\ar@{-}[uuurrr] \ar@{-}[uuulll] \ar@{-}[d] &&&\\
&&&&&&
}}}

\def\arbreACA{\vcenter{\xymatrix@R=1pt@C=1pt{
&&&&&&\\
&*{}\ar@{-}[ur] &&&&*{}\ar@{-}[ul] & \\
&&&&&&\\
&&&*{}\ar@{-}[uuurrr] \ar@{-}[uuulll] \ar@{-}[d] &&&\\
&&&&&&
}}}

\def\arbreCAB{\vcenter{\xymatrix@R=1pt@C=1pt{
&&&&&&\\
&&&*{}\ar@{-}[ur] &&& \\
&&&&*{}\ar@{-}[uull] &&\\
&&&*{}\ar@{-}[uuurrr] \ar@{-}[uuulll] \ar@{-}[d] &&&\\
&&&&&&
}}}

\def\arbreCBA{\vcenter{\xymatrix@R=1pt@C=1pt{
&&&&&&\\
&&&&&*{}\ar@{-}[ul] & \\
&&&&*{}\ar@{-}[uull] &&\\
&&&*{}\ar@{-}[uuurrr] \ar@{-}[uuulll] \ar@{-}[d] &&&\\
&&&&&&
}}}

\def\arbreABCD{\vcenter{\xymatrix@R=1pt@C=1pt{
&&&&&&&&\\
&*{}\ar@{-}[ur] &&&&&&& \\
&&*{}\ar@{-}[uurr] &&&&&&\\
&&&*{}\ar@{-}[uuurrr] &&&&&\\
&&&&*{}\ar@{-}[uuuurrrr] \ar@{-}[uuuullll] \ar@{-}[d] &&&&\\
&&&&&&
}}}

\def\arbreACAD{\vcenter{\xymatrix@R=1pt@C=1pt{
&&&&&&&&\\
&*{}\ar@{-}[ur] &&&&*{}\ar@{-}[ul] &&& \\
&&&&&&&&\\
&&&*{}\ar@{-}[uuurrr] &&&&&\\
&&&&*{}\ar@{-}[uuuurrrr] \ar@{-}[uuuullll] \ar@{-}[d] &&&&\\
&&&&&&
}}}

\def\arbreBADA{\vcenter{\xymatrix@R=1pt@C=1pt{
&&&&&&&&\\
&&&*{}\ar@{-}[ul] &&&&*{}\ar@{-}[ul] & \\
&&*{}\ar@{-}[uurr] &&&&&&\\
&&& &&&&&\\
&&&&*{}\ar@{-}[uuuurrrr] \ar@{-}[uuuullll] \ar@{-}[d] &&&&\\
&&&&&&
}}}

\def\Kzero{\xymatrix@R=4pt@C=4pt{
\\
\\
\\
{\bullet}\\
}}

\def\KzeroC{\xymatrix@R=4pt@C=4pt{
\\
\\
(1)\\
{\bullet}\\
}}

\def\KunA{\xymatrix@R=4pt@C=4pt{
&&\\
&&\\
&&\\
*{}\ar@{-}[rr]&&*{}\\
}}

\def\KunC{\xymatrix@R=4pt@C=4pt{
&&\\
&&\\
(1,2)&&(2,1)\\
*{}\ar@{-}[rr]&&*{}\\
}}
\def\TdeuxA{\xymatrix@R=4pt@C=4pt{
&&*{}\ar@{-}[dddll]\ar@{-}[dddrr]&&\\
&&&&\\
&&&&\\
*{}\ar@{-}[rrrr]&&&&*{}\\
}}

\def\GdeuxA{\xymatrix@R=4pt@C=4pt{
&&&&\\
&&&&\\
*{}\ar@{-}@/^1pc/[rrrr]\ar@{-}@/_1pc/[rrrr]&&&&*{}\\
&&&&\\
&&&&\\
}}

\def\TtroisA{\xymatrix@R=4pt@C=4pt{
&&&*{}\ar@{-}[ddddlll]\ar@{-}[ddddrrr]\ar@{-}[drrr]&&&\\
&&&&&&*{}\ar@{-}[ddd]\ar@{.}[dddllllll]\\
&&&&&*{}&\\
&&&&&&\\
*{}\ar@{-}[rrrrrr]&&&&&&*{}\\
}}

\def\GtroisA{\xymatrix@R=4pt@C=4pt{
&&&*{}\ar@{-}@/^0.5pc/[ddrrr]&&&\\
&&&&&&\\
*{}\ar@{.}@/^1pc/[rrrrrr]\ar@{-}@/_1pc/[rrrrrr]\ar@{-}@/^0.5pc/[uurrr]\ar@{-}@/_0.5pc/[ddrrr]&&&&&&*{}\\
&&&&&&\\
&&&*{}\ar@{-}@/_0.5pc/[uurrr]&&&\\
}}

\def\TquatreA{\xymatrix@R=4pt@C=4pt{
&&&*{}\ar@{-}[ddddlll]\ar@{-}[ddddrrr]\ar@{-}[drrr]&&&\\
&&&&&&*{}\ar@{-}[ddd]\ar@{.}[dddllllll]\\
&&&&&*{}&\\
&&&&*{}\ar@{.}[dllll]\ar@{.}[uuul]\ar@{.}[uurr]\ar@{.}[drr]&&\\
*{}\ar@{-}[rrrrrr]&&&&&&*{}\\
}}

\def\QdeuxA{\xymatrix@R=3pt@C=3pt{
*{}\ar@{-}[rrrr]\ar@{-}[dddd]&&&&*{}\ar@{-}[dddd]\\
&&&&\\
&&&&\\
&&&&\\
*{}\ar@{-}[rrrr]&&&&*{}\\
}}

\def\QtroisA{\xymatrix@R=5pt@C=5pt{
&&*{}\ar@{-}[rrr]\ar@{.}[ddd]&&&*{}\ar@{-}[ddd]\\
*{}\ar@{-}[rrr]\ar@{-}[ddd]\ar@{-}[urr]&&&*{}\ar@{-}[ddd]\ar@{-}[urr]&&\\
&&&&&\\
&&*{}\ar@{.}[rrr]&&&*{}\\
*{}\ar@{-}[rrr]\ar@{.}[urr]&&&*{}\ar@{-}[urr]&&\\
}}

\def\KdeuxA{\xymatrix@R=6pt@C=6pt{
&&&&&\\
&&*{}\ar@{-}[dll]\ar@{-}[ddrrr]  &&& \\
*{}\ar@{-}[dd] &&&&&\\
&&&&&*{}\ar@{-}[ddlll]  \\
*{}\ar@{-}[drr] &&&&& \\
&&*{}&&& \\
}}

\def\KdeuxC{\xymatrix@R=4pt@C=4pt{
&&&(1,2,3)&&&\\
&&&*{}\ar@{-}[dll]\ar@{-}[ddrrr]  &&&& \\
(2,1,3)&*{}\ar@{-}[dd] &&&&&&\\
&&&&&&*{}\ar@{-}[ddlll]&(1,4,1)  \\
(3,1,2))&*{}\ar@{-}[drr] &&&&&& \\
&&&*{}&&&& \\
&&&(3,2,1)&&&&
}}

\def\PdeuxC{\xymatrix@R=2pt@C=2pt{
&&&(1,2,3)&&\\
&&&*{}\ar@{-}[ddll]\ar@{-}[ddrr]  &&& \\
&&&&&&\\
(2,1,3)&*{}\ar@{-}[dd] &&&&*{}\ar@{-}[dd] &(1,3,2)\\
&&&&&  \\
(3,1,2)&*{}\ar@{-}[ddrr] &&&&*{}\ar@{-}[ddll]&(2,3,1) \\
&&&&&&\\
&&&*{}&&& \\
&&&(3,2,1)&&&
}}

\def\PdeuxA{\xymatrix@R=4pt@C=4pt{
&&&&&\\
&&*{}\ar@{-}[dll]\ar@{-}[drr]  &&& \\
*{}\ar@{-}[dd] &&&&*{}\ar@{-}[dd] &\\
&&&&& \\
*{}\ar@{-}[drr] &&&&*{}\ar@{-}[dll] & \\
&&*{}&&& \\
}}

\def\KtroisA{\xymatrix@R=2pt@C=2pt{
&*{}\ar@{-}[rrrrrrrrrr] *{}\ar@{-}[rrrrddd] *{}\ar@{-}[ldd] &&&& &&& &&&*{}\ar@{.}[llldddddd]  *{}\ar@{-}[rrd] && \\
&&&& &&& &&& &&&*{}\ar@{-}[llldd] *{}\ar@{-}[llldddddd]   \\
*{}\ar@{-}[rrrrddd]  *{}\ar@{-}[rrdddd] &&&&& &&& &&& && \\
&&&&&*{}\ar@{-}[rrrrr] *{}\ar@{-}[ldd]  &&& &&*{}\ar@{-}[lldddd] & && \\
&&&&& &&& &&& && \\
&&&&*{}\ar@{-}[rdd] & &&& &&& & &\\
&&*{}\ar@{.}[rrrrrr] *{}\ar@{-}[rrrddd] &&& &&&*{}\ar@{.}[rrd]  &&& && \\
&&&&&*{}\ar@{-}[rrr] *{}\ar@{-}[dd]  &&&*{}\ar@{-}[dd]  &&*{}\ar@{-}[lldd] & & &\\
&&&&& &&& &&& && \\
&&&&&*{}\ar@{-}[rrr]  &&&*{} &&& && \\
}}

\def\KtroisF{\xymatrix@R=0pt@C=0pt{
&*{}\ar[rrrrrrrrrr] *{}\ar@{<-}[rrrrddd] *{}\ar[ldd] &&&& &&& &&&*{}\ar@{.>}[llldddddd]  *{}\ar@{<-}[rrd] && \\
&&&& &&& &&& &&&*{}\ar@{<-}[llldd] *{}\ar[llldddddd]   \\
*{}\ar@{<-}[rrrrddd]  *{}\ar[rrdddd] &&&&& &&& &&& && \\
&&&&&*{}\ar[rrrrr] *{}\ar[ldd]  &&& &&*{}\ar[lldddd] & && \\
&&&&& &&& &&& && \\
&&&&*{}\ar[rdd] & &&& &&& & &\\
&&*{}\ar@{.>}[rrrrrr] *{}\ar@{<-}[rrrddd] &&& &&&*{}\ar@{<.}[rrd]  &&& && \\
&&&&&*{}\ar[rrr] *{}\ar[dd]  &&&*{}\ar[dd]  &&*{}\ar@{<-}[lldd] & & &\\
&&&&& &&& &&& && \\
&&&&&*{}\ar[rrr]  &&&*{} &&& && \\
}}

\def\KtroisFOld{\xymatrix@R=3pt@C=3pt{
&*{}\ar[rrrrrrrrrr] *{}\ar@{<-}[rrrrddd] *{}\ar[ldd] &&&& &&& &&&*{}\ar@{.>}[llldddddd]  *{}\ar@{<-}[rd] & \\
&&&& &&& &&& &&*{}\ar@{<-}[lldd] *{}\ar[llldddddd]   \\
*{}\ar[rrrrddd]  *{}\ar@{-}[rrdddd] &&&&& &&& &&& & \\
&&&&&*{}\ar[rrrrr] *{}\ar[ldd]  &&& &&*{}\ar@{<-}[lldddd] & & \\
&&&&& &&& &&& & \\
&&&&*{}\ar[rdd] & &&& &&& & \\
&&*{}\ar@{.>}[rrrrrr] *{}\ar@{<-}[rrrddd] &&& &&&*{}\ar@{<.}[rd]  &&& & \\
&&&&&*{}\ar[rrr] *{}\ar[dd]  &&&*{}\ar[dd]  &*{}\ar@{<-}[ldd] && & \\
&&&&& &&& &&& & \\
&&&&&*{}\ar[rrr]  &&&*{} &&& & \\
}}

\def\PtroisA{\xymatrix@R=4pt@C=4pt{
&&&*{}\ar@{-}[dll]\ar@{.}[dr]\ar@{-}[r] &*{}\ar@{-}[drr]\ar@{-}[dll] &&& \\
&*{}\ar@{-}[ddl]\ar@{-}[r]&*{}\ar@{-}[dd] & &*{}\ar@{.}[ddl]\ar@{.}[drr]&&*{}\ar@{-}[dd]\ar@{-}[dr]& \\
&&& & &&*{}\ar@{.}[r]\ar@{.}[ddl]&*{}\ar@{-}[dd] \\
*{}\ar@{-}[dd]\ar@{.}[dr]&&*{}\ar@{-}[ddl]\ar@{-}[drr]&*{}\ar@{.}[dll]\ar@{.}[drr] & &&*{}\ar@{-}[dll]\ar@{-}[dr]& \\
&*{}\ar@{.}[dd]&& &*{}\ar@{-}[ddl] &*{}\ar@{.}[dd]&&*{}\ar@{-}[ddl] \\
*{}\ar@{-}[dr]\ar@{-}[r]&*{}\ar@{-}[drr]&& & &&& \\
&*{}\ar@{-}[drr]&&*{}\ar@{-}[dr] & &*{}\ar@{.}[dll]\ar@{.}[r]&*{}\ar@{-}[dll]& \\
&&&*{}\ar@{-}[r] &*{} &&& \\
}}

\def\TamariDeux{\xymatrix@R=15pt@C=10pt{
300&&210&&120\\
*{\times} \ar@{-}[rrrr]\ar@{-}[rdd]&&*{\times} && *{\times}\ar@{-}[ldd]\\
&&&&\\
&*{\times} \ar@{-}[rr]&&*{\times}&&\\
& 201 && 111 &
}}

\def\TamariTrois{\xymatrix@R=15pt@C=10pt{
&*{}\ar@{-}[rrrrrrrr]\ar@{-}[ddddrr]\ar@{-}[drr]&&&*{}\ar@{.}[rdd]&&   &*{}\ar@{-}[dll]&&*{}\ar@{-}[ddddll]*{}\ar@{-}[lld] \\
&&&*{}\ar@{-}[rrrr]\ar@{-}[ddr]&&*{}   &&*{}\ar@{-}[ddl]&&  \\
&&*{}\ar@{.}[rrrrrr]&&&*{}   &&&&  \\
&&&&*{}\ar@{-}[rr]*{}\ar@{-}[dl]&   &*{}\ar@{-}[rd]&&&  \\
&&&*{}\ar@{-}[rrrr]&&   &&*{}&&  
}}

\def\KtroisFgrand{\xymatrix@R=4pt@C=4pt{
&*{}\ar[rrrrrrrrrr] *{}\ar@{<-}[rrrrddd] *{}\ar[ldd] &&&& &&& &&&*{}\ar@{.>}[llldddddd]  *{}\ar@{<-}[rrd] && \\
&&&& &&& &&& &&&*{}\ar@{<-}[llldd] *{}\ar[llldddddd]   \\
*{}\ar@{<-}[rrrrddd]  *{}\ar[rrdddd] &&&&& &&& &&& && \\
&&&&&*{}\ar[rrrrr] *{}\ar[ldd]  &&& &&*{}\ar[lldddd] & && \\
&&&&& &&& &&& && \\
&&&&*{}\ar[rdd] & &&& &&& & &\\
&&*{}\ar@{.>}[rrrrrr] *{}\ar@{<-}[rrrddd] &&& &&&*{}\ar@{<.}[rrd]  &&& && \\
&&&&&*{}\ar[rrr] *{}\ar[dd]  &&&*{}\ar[dd]  &&*{}\ar@{<-}[lldd] & & &\\
&&&&& &&& &&& && \\
&&&&&*{}\ar[rrr]  &&&*{} &&& && \\
}}

\def\t{\otimes}

\def\d{\succ}
\def\g{\prec}
\def\l{\dashv}
\def\r{\vdash}

\def\aa{\alpha}

\def\PP{{\mathcal P}}

\def\RRR{{\mathbb R}}

\def\KKK{{\mathcal K}}

\def\Ai{A_{\infty}}

\makeindex             

\begin{document}

\setcounter{chapter}{3}

\title*{Dichotomy of the addition of natural numbers}
\author{Jean-Louis Loday}
\institute{Jean-Louis Loday \at Institut de Recherche Math\'{e}matique Avanc\'{e}e,
CNRS et Universit\'{e} de Strasbourg,
7 rue Ren\'{e}-Descartes,
67084 Strasbourg Cedex, France,  
\email{loday@math.unistra.fr} }
%
%
\maketitle

\abstract*{}

\abstract{This is an elementary presentation of the arithmetic of trees. We show how it is related to the Tamari poset. In the last part we investigate various ways of realizing this poset as a polytope (associahedron), including one inferred from Tamari's thesis. }

\section*{Introduction}

In this paper the addition of integers is split into two operations which satisfy some relations. These relations are taken so that they split the associativity relation of addition into three. Under these new operations, the unit $1$ generates elements which are in bijection with the planar binary rooted trees. More precisely, any integer $n$ splits as the disjoint union of the trees with $n$ internal vertices. The Tamari poset is a partial order structure on this set of trees. We show how the addition on trees is related to the Tamari poset structure. This first part is an elementary presentation of results contained in ``Arithmetree'' \cite{Loday02} by the author and in \cite{LodayRonco02} written jointly with M.\ Ronco. 

Prompted by an unpublished page of Tamari's thesis, we investigate various ways of realizing the Tamari poset as a polytope. In particular we show that Tamari's way of indexing the planar binary rooted trees gives rise to a hypercube-like polytope on which the associahedron is drawn.

\section{About the formula $1+1=2$}

The equality $3+5=8$ can be seen either as $3$ acting on the left on $5$, or as $5$ acting on the right on $3$. Since adding $3$ and $5$ is both, one can imagine to  ``split'' this sum into two pieces reflecting this dichotomy. Physically, splitting the addition symbol $+$ into two pieces gives:
\begin{gather*}
+\\
\l \r\\
\l\quad \r
\end{gather*}
that is, the symbols  $\l$ and $\r$. Hence, since $1+1=2$, one defines  two new elements 
$1\l 1$ and $1\r 1$ so that  
$$1\l 1\  \cup\  1\r 1= 1+1=2.$$ 

\section{Splitting the integers into pieces}

How to go on  ? A priori one can form eight  elements out of three copies of $1$ and of the operations \emph{left} $\l$ and \emph{right} $\r$, that is 
\begin{gather*}
 (1 \l 1) \l 1 \quad , \quad   (1 \l 1) \r 1 \quad , \quad  (1 \r 1) \l 1 \quad , \quad  (1 \r 1) \r 1 \quad , \quad \\
1 \l (1 \l 1) \quad , \quad   1 \l (1 \r 1) \quad , \quad  1 \r (1 \l 1) \quad , \quad  1 \r (1 \r 1) \quad . \quad 
  \end{gather*}
  
But we would like to keep associativity  of the operation $+$, so we want that the union of the elements on the first row is equal to the union of the elements of the second row. More generally for any component $r$ and $s$ we split the sum as 
$$r+s = r\l s \ \cup \ r \r s \ .$$
Taking again our metaphore of left action and right action, it is natural to choose the relations
\begin{align*}
&(*)\qquad\begin{cases}
(r \l s) \l t = r \l (s +t), \\
(r \r s) \l t = r \r (s \l t), \\
(r +s) \r t = r \r (s \r t), \\
\end{cases}
\end{align*}

\noindent since, by taking the union, we get readily $(r+s)+t = r+(s+t).$ 
The first relation says that  ``acting on the right by $s$  and then by  $t$'' is the same as ``acting by  $s+t$''. (The kowledgeable reader will remark the analogy with bimodules). Since we have three relations, our eight elements in the case  $r=s=t=1$ go down to  five, which are the following:
$$(1\r 1)\r 1\ , \ (1\l 1)\r 1\ , \ 1\r 1 \l 1 \ , \ 1\l (1\r 1)\ , \ (1\l1)\l 1 \ .$$

Indeed, since one has $(1 \r 1) \l 1= 1 \r (1 \l 1)$,  the parentheses can be discarded. On the other hand the two elements $1\r (1\r 1)$ and $(1\l 1)\l 1$  can be written respectively:
\begin{eqnarray*}
1\r (1\r 1) &=& (1\r 1)\r 1\ \cup \ (1\l 1)\r 1\  ,\\
 \  (1\l 1)\l 1&=& 1\l (1\r 1) \ \cup \ 1\l (1\l 1).
\end{eqnarray*}

In conclusion, we have decomposed the integer $2$ into two components $1\l 1$ and $1\r 1$ and the integer $3$ into five components (see above), the integer $1$ has only one component, namely itself. How about the integer $n$ ? In fact, not only would we like to decompose  $n$ into the union of some components, but we would also like to know how to add these components. The test will consist in checking that adding the components of $n$ with all the components of $m$, we get back the union of all the components of $m+n$.

\section{Trees and addition on trees}\label{trees} In order to understand the solution we introduce the notion of \emph{planar binary rooted tree}, that we simply call tree in the sequel. Here are the first of them:

\begin{displaymath}
PBT_1 = \{\  \vert \ \}\  , \quad PBT_2=\big\{ \arbreA\big\}\  , \quad 
PBT_3=\Big\{\arbreBA\ ,\  \arbreAB \Big\}\ ,
\end{displaymath}
\begin{displaymath}
{PBT_4=\bigg\{ \arbreABC}\ ,\  {{}\arbreBAC },\  {{}\arbreACA },\  {{}\arbreCAB },\  {{}\arbreCBA }\bigg\}\ .
\end{displaymath}

Such a tree $t$ is completely determined by its left part $t^{l}$ and its right part $t^{r}$, which are themselves trees. The tree $t$ is obtained  by joining the roots of $t^{l}$ and of $t^{r}$ to a new vertex and adding a root:
$$\xymatrix@R=10pt@C=10pt{ 
&t^{l}\ar@{-}[dr] && t^{r}\ar@{-}[dl] \\
t=&&*{}\ar@{-}[d] &\\
&&*{}&
}$$
 This construction is called the \emph{grafting} of  $t^{l}$ and $t^{r}$. One writes $t= t^{l}\vee t^{r}$. 

Hence any nontrivial tree (that is different from $\vert$) is obtained from the trivial tree $\vert$ by iterated grafting. The elements of the set $PBT_{n}$ are the trees with $n$ leaves that is with $n-1$ internal vertices. The number of elements in $PBT_{n}$ is the Catalan number $c_{n-1}$. It is known that  $c_{n}= \frac{(2n)!}{n!\, (n+1)!}= \frac{1}{n+1}\binom{2n}{n}$. 

The solution to the splitting of natural numbers is going to be a consequence of the properties of the operations  $\l$ and $\r$ on trees, which are defined as follows. 
For any nontrivial trees $s$ and $t$ one defines  recursively the two operations $\l$ and $\r$ by the formulas
$$(\ddag)\qquad  s\l t := s^{l}\vee (s^{r}+t)\quad, \quad s\r t := (s+t^{l})\vee t^{r},$$
and the sum by 
$$s+t :=s\l t \ \ \cup \ s\r t .$$
 The trivial tree is supposed to be a neutral element for the sum:  $\vert = 0$, so $s\l 0= s$ and $0\r t = t$.
The unique tree with one internal vertex (Y shape tree) represents $1$. Then one gets 
$$ \arbreA \l \arbreA = \arbreBA \quad , \quad  \arbreA \r \arbreA = \arbreAB \ .$$ 
Notice the matching between the orientation of the leaves and the involved operations: the middle leaf of the tree $\arbreBA$, resp.\ $\arbreAB$,  is oriented to the left, resp.\ right,  and this tree represents the element $1\l 1$, resp. $1\r 1$.

The principal properties of these two operations are given by the following statement.

\begin{proposition} \cite{LodayRonco02} The operations $\l$ and $\r$ satisfy the relations $(*)$. Any tree can be obtained from the initial tree $\arbreA$ by iterated application of the operations left and right.
\end{proposition}

The solution is then the following. The integer $n$ is the disjoint union of the elements of $PBT_{n+1}$, that is the trees with  $n$ internal vertices. For instance:
\begin{eqnarray*}
0&=& \vert\\
1&=& \arbreA \\
2 &=& \arbreAB \cup \arbreBA\\
3&=& \arbreABC \cup \arbreBAC \cup \arbreACA \cup \arbreCAB \cup \arbreCBA \\
\end{eqnarray*}

The sum $+$ of integers can be extended to the components of these integers, that is to trees. Even better, the operations left and right can be extended to the trees. The above formulas $(\ddag)$ give the algorithm to perform the computation.

\section{Where we show that  $1+1=2$ and $2+1=3$}
Here are some computation examples:
\begin{eqnarray*}
1\l 1 &=& \arbreA \l \arbreA = \arbreBA\quad , \quad 1\r 1 = \arbreA\r \arbreA = \arbreAB, \\
2\l 1 &=& \Big(\arbreAB \cup \arbreBA \Big)\l \arbreA =   \arbreAB  \l \arbreA  \cup \arbreBA \l \arbreA \\
&=& \arbreACA \cup \arbreCAB \cup \arbreCBA,\\
2\r 1 &=& \Big(\arbreAB \cup \arbreBA \Big)\r \arbreA =   \arbreAB  \r \arbreA  \cup \arbreBA \r \arbreA \\
&=&  \arbreABC \cup \arbreBAC.\\
\end{eqnarray*}

Notice that  $1+1 = 1\l 1 \ \cup \ 1\r 1 = 2$ and that $2+1 = 2\l 1 \ \cup \ 2\r 1 = 3$, since $3$ is the union of the five trees of $PBT_{4}$. They represent the five elements which can be written with three copies of $1$ (see above).  Similarly one can check that
\begin{eqnarray*}
m+n &=& m\l n \ \cup \ m\r n\\
&=&\Big(\bigcup_{s\in PBT_{n+1}}s\ \l \ \bigcup_{t\in PBT_{m+1}}t\Big)\ \bigcup \ \Big(\bigcup_{s\in PBT_{n+1}}s\ \r \ \bigcup_{t\in PBT_{m+1}}t\Big) \\
&=&\Big(\bigcup_{s\in PBT_{n+1}, t\in PBT_{m+1}}s\ \l \ t\Big)\ \bigcup \ \Big(\bigcup_{s\in PBT_{n+1}, t\in PBT_{m+1}}s\ \r \ t\Big) \\
&= & \bigcup_{r\in PBT_{m+n+1}}r = m+n.
\end{eqnarray*}

Finally there are two ways to look at the  $c_{n+1}$ components of the integer $n$: either through trees, or through $n$ copies of $1$ and the operations $\l$ and $\r$ duly parenthesized. Recall that this second presentation is not unique.

\bigskip

There are many families of sets whose number is the Catalan number (they are called Catalan sets). For each of them one can translate the algebraic structure unraveled above. In \cite{AvalViennot10} Aval and Viennot have performed this task for the ``alternative Catalan tableaux''. It is interesting to notice that, in their description of the sum of two tableaux (see loc.\ cit.\ p.\ 6), there are two different kinds of tableaux. In fact the tableaux of one kind give the left operation and the tableaux of the other kind give the right operation.

\section{The integers as molecules}\label{molecule} Let us think of the integers as molecules and of its components as atoms. Then one would like to know of the ways the atoms are bonded in order to form the molecule. Since the molecule $2$ has only two atoms, we pretend that there is a bond between the two atoms: 

$$\arbreAB \xymatrix{\ar@{-}[rr]&& }\arbreBA$$

For our mathematical purpose it is important to see this bond as an oriented relation (it is called a covering relation):

$$\arbreAB \xymatrix{\ar[rr]&& }\arbreBA$$

For the molecule $n$ one puts a bond between two atoms (i.e.\ trees) whenever one can obtain one of them from the other one by a local change as in the molecule $2$ case. 
Here is what we get for  $n=3$:

\begin{figure}[h]
\centering
\includegraphics[scale=0.10]{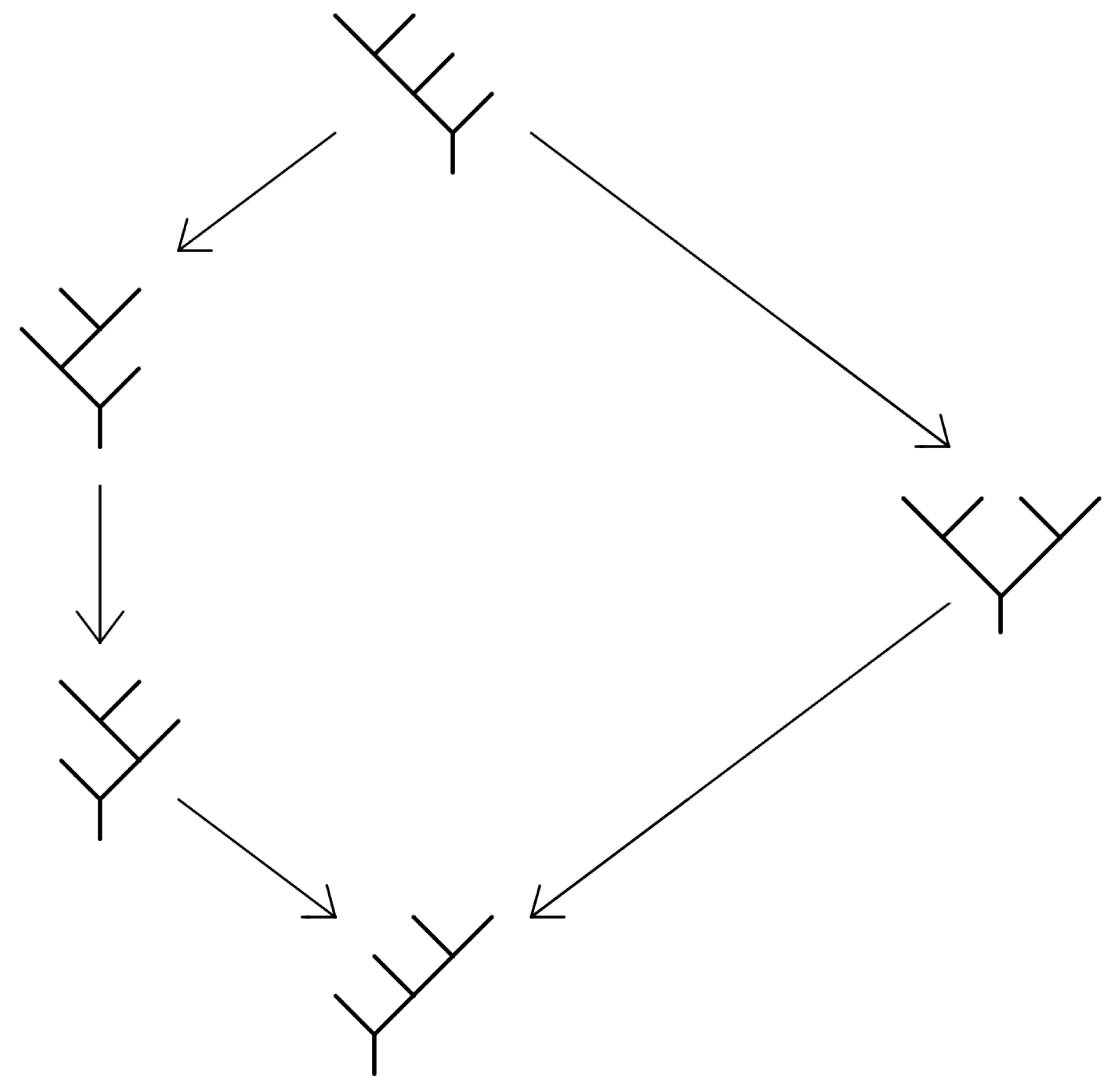}
\end{figure}
and for $n=4$ (without mentioning the trees):

$$\KtroisFgrand$$

\bigskip

These drawings already appeared, under slightly different shape,  in Dov Tamari's original thesis, defended in 1951 (see the discussion below in section \ref{realization}). In fact the covering relations on the set $PBT_{n}$ make it into a  ``partially ordered set'', usually abbreviated into ``poset''. This is the \emph{Tamari poset} on trees \cite{Tamari51}. The reason for introducing this poset at this point is its strong relationship with the algebraic structure that we just described on trees. It is given by the following statement proved in a joint work with Mar\'\i a Ronco.

\begin{theorem} \cite{LodayRonco02} The sum of the trees $t$ and $s$ is the union of all the trees which fit in between $t/s$ and  $t\backslash s$ :
$$t + s = \bigcup_{t/s \leq r \leq t\backslash s} r$$
where $t/s$ is obtained by grafting the root of $t$ on the leftmost leaf of $s$ and $t\backslash s$ is obtained by grafting the root of $s$ on the rightmost leaf of $t$.
\end{theorem}

This formula makes sense because one can prove (cf.\ loc.\ cit.) that, for any trees $t$ and $s$,  we have $t/s \leq t\backslash s$.

 \section{Multiplication of trees}
 
 Multiplication of natural numbers is obtained from addition since:
 $$n \times m := \underbrace{m+ \cdots + m}_{n \textrm{ copies}} .$$
In other words, one writes $n$ in terms of the generator $1$:
  $$n  = \underbrace{1+ \cdots + 1}_{n \textrm{ copies}} ,$$
and then one replaces $1$ by $m$ everywhere to obtain $n�\times m$.
  
  The very same process enables us to define the product $t\times s$ of the trees $t$ and $s$. First we write $t$ in terms of $1$ with the help of the left and right operations, and then we replace each occurence of $1$ by $s$. Here are some examples:
  
  $$\arbreAB \times \arbreBA= \arbreBADA$$
  $$\arbreAB \times \arbreAB= \arbreACAD \cup \arbreABCD.$$
  
  The proof of the first case is as follows. Since we have  $\arbreAB= 1 \r 1$, we can write:
  $$\arbreAB \times \arbreBA=\arbreBA \r \arbreBA = \arbreBA \vee \arbreA = \arbreBADA .$$
 It is immediate to check that if $s$ has $n$ internal vertices and $t$ has $m$ internal vertices, then $s \times t$ has $nm$ internal vertices. Another relationship with the product of natural numbers is the following.  Replacing $n$ by the union of trees of $PBT_{n+1}$, and $m$ by the union of trees of $PBT_{m+1}$, then $n\times m$ is actually the union of all the trees with $nm$ internal vertices.
  
  Some of the properties of the multiplication are preserved, but not all. The associativity holds and the distributivity with respect to the left factor also holds. But right distributivity does not (and of course commutativity does not hold). This is the price to pay for such a generalization. More properties and computation can be found in  \cite{Loday02}. The interesting paper \cite{BrunoYazaki08} deals with the study of prime numbers (we should say prime trees) in this framework.  
  
  Let us summarize the properties of the sum and the product of trees versus the sum and the product of integers. We let $\PP(PBT)$ be the set of non-empty subsets of $PBT_{n}$ for all $n$. 
  
  \begin{proposition} There are maps
  $$\NN \mono \PP(PBT) \epi \NN$$
  which are compatible with the sum and the product. The composite is the identity.
  \end{proposition}
  Indeed, the first map sends $n$ to the union of all the trees in $PBT_{n}$. The second map sends a subset to the arity of its components.

  \section{Trees and polynomials} The algebra of polynomials (let us say with real coefficients) in one variable $x$  admits the monomials $x^{n}$ for basis. Since we know how to decompose an integer into the union of trees, we dare to write
$$x^{n}= \sum_{t\in PBT_{n+1}} x^{t}.$$

More specifically,
we consider the vector space spanned by the elements $x^{t}$ for any tree $t$. 
As usual, the sum of exponents gives rise to a product of factors:
$$x^{n+m}= x^{n}x^{m}\quad , \quad x^{s+t}= x^{s}x^{t},$$
where $s$ and $t$ are trees. We use the notation   $x^{\vert}= x^{0}= 1$ and 
$x^{\!\!\petitarbreA}=x^{1}=x$.

In fact there is no reason to consider only polynomials and one can as well consider series since the sum and the product are well-defined.

In this framework  the operations 
\begin{eqnarray*}
 \begin{array}{r} \mbox{union} \\ \mbox{addition} \\ \mbox{multiplication} \end{array}
 \quad \mbox{on trees} \qquad \mbox{become} \qquad 
 \begin{array}{r} \mbox{addition} \\ \mbox{multiplication} \\ \mbox{composition} \end{array} 
  \quad \mbox{on polynomials} \, .
\end{eqnarray*}
What about the operations $\l$ and $\r$ ? They give rise to two operations denoted $\g$ and $\d$ respectively, on polynomials. These two operations are bilinear and satisfy the relations:

\begin{displaymath}
\begin{array}{rcl}
(r \g s) \g t &=& r \g (s \g t + s\d t) \, , \\
(r \d s) \g t &=& r \d (s \g t) \, , \\
(r \g s+ r \d s) \d t &=& r \d (s \d t) \, . \\
\end{array}
\end{displaymath}

A vector space  $A$ endowed with two bilinear operations  $\g$ and $\d: A\t A \to A$, satisfying the relations just mentioned, is called a \emph{dendriform algebra}, cf.\ \cite{Loday95, Loday01}.  

The dendriform algebras show up in many topics in mathematics: higher algebra \cite{BurgunderRonco10, Dotsenko09, Ronco00}, homological algebra \cite{Loday12}, combinatorial algebra \cite{AvalViennot10, AguiarSottile06, Devadoss09, NovelliThibon07, PilaudSantos09, Postnikov09}, algebraic topology \cite{Chapoton02, Vallette08, Yau07}, renormalization theory \cite{Brouder04, BrouderFrabetti03}, quantum theory \cite{GGL}, to name a few. It is closely related to the notion of shuffles. In fact it could be called the theory of ``non-commutative shuffles''.

\section{Realizing of the associahedron}\label{realization}

In Dov Tamari's seminal work ``Monoides pr\' eordonn\' es et cha\^ {\i}nes de Malcev'' \cite{Tamari-these}, which is his French doctoral thesis defended in 1951, the picture displayed in 
Fig.~\ref{fig:Tthesis} appears on page 12.
\begin{figure}[h]
\begin{center}
\scalebox{0.42}{\includegraphics{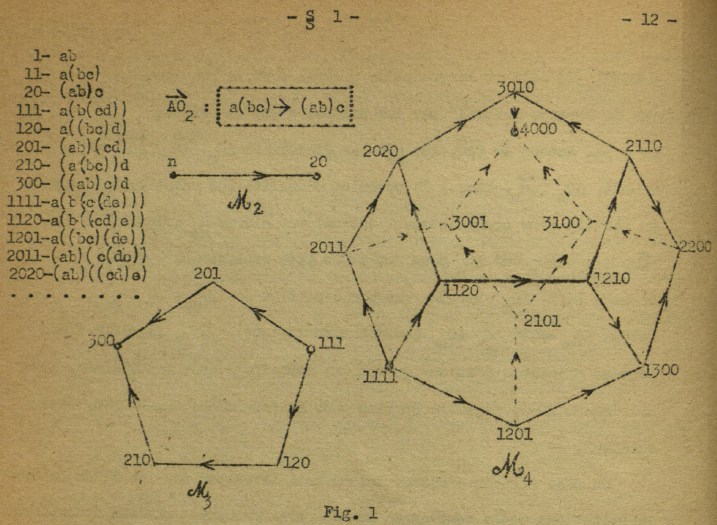}}
\parbox{5.5cm}{
\caption{Excerpt from Tamari's thesis. \label{fig:Tthesis}}
}
\end{center}
\end{figure}
Unfortunately this part has not been reproduced in the published text \cite{Tamari51} and therefore has been forgotten for all these years. It is very interesting on three grounds. First, it is the first appearance of the \emph{Tamari poset}. Second, the Tamari poset is portrayed in dimension 2 and 3 as a polygon and a polyhedron respectively. Third, the parenthesizings has been replaced by a code that one can consider as coordinates in the euclidean space. We now analyze these three points.

\subsection{Tamari poset} The Tamari poset, appearing often as the \emph{Tamari lattice} in the literature (since it is a lattice), proved to be helpful in many places in mathematics. I mentioned earlier in this text its relevance with dendriform structures. It is playing a key role in the problem of endowing the tensor product of $\Ai$-algebras with an $\Ai$-algebra structure, cf.\ \cite{Loday12} for two reasons. First, the Tamari poset gives rise to a cell complex called the associahedron or the Stasheff polytope (see below). In 1963 Jim Stasheff showed that it encodes the notion of ``associative algebra up to homotopy'', now called $\Ai$-algebras.  Let us recall that such an algebra $A$ is equipped with a $k$-ary operation $m_{k}: A^{\t k} \to A, k\geq 2,$ which satisfy some universal relations describing the topological structure of the associahedron. The second reason comes as follows. For a fixed integer $n$, the associahedron $\KKK^{n}$ is a cell complex of dimension $n$. We can prove that its cochain complex $C^{\bullet}(\KKK^{n})$ can be endowed with a structure of $\Ai$-algebra. The operations $m_{k}$ can be made explicit in terms of the Tamari poset relation,  cf.\ \cite{Loday12}.

\subsection{Associahedron and regular pentagons} The sentence following Fig.~\ref{fig:Tthesis} in Tamari's 
thesis is the following

``G\' en\' eralement, on aura des hyperpoly\`edres.''

But no further information is given. In fact, as we know now, we can realize the Tamari poset as a convex polytope so that each element of the poset is a vertex and each covering relation is an edge (see below). There is no harm in taking the regular pentagon in dimension 2. However, in contrast to what Fig.~\ref{fig:Tthesis} suggests, one cannot realize the associahedron in dimension 3 with regular pentagons. What happens is the following: the four vertices corresponding to the parenthesizings $2020, 2011, 1120$ and $1111$ do not lie in a common plane. It is a good trigonometric exercise for first year undergraduate students. If we take the convex hull of $\mathcal{M}_{4}$ 
of Fig.~\ref{fig:Tthesis} (that is keeping regular pentagons), then the faces are made up of 6 pentagons, and 6 triangles instead of the 3 quadrangles. There are 3 edges which show up and which do not correspond to any covering relation:

$$2011 - 1120, \quad 30012 - 3100, \quad 2200 -1210.$$

\subsection{Realizations of the associahedron} Though Tamari does not mention it, we can think of his clever way of indexing the parenthesizings  as coordinates of points in the euclidean space $\RRR^{n+1}$. Let us recall briefly his method: given a parenthesizing (which is equivalent to a planar binary tree $t$) of the word $x_{0}x_{1}\ldots x_{n+1}$ we count the number of opening parentheses in front of $x_{0}$, then $x_{1}$, etc., up to $x_{n}$. For instance the word $((x_{0}x_{1}) x_{2})$ gives $2\ 0$ and the word $(x_{0}(x_{1} x_{2}))$ gives $1\ 1$. Let us denote this sequence of numbers by 
$$M(t)=(\aa_{0}, \ldots, \aa_{n})\in \RRR^{n+1}.$$
Since the number of parentheses depends only on the length of the word, we have $\sum\aa_{i}= n+1$ and the points $M(t)$  lie in a common hyperplane.  What does the convex hull look like ? In dimension 2 we get the following pentagon:

$$\TamariDeux$$

As we see it is a quadrangle (that is a deformed square) with one point added on an edge.

In dimension 3 we get:

$$\TamariTrois$$

\noindent that is a deformed cube on which the associahedron has been drawn. In order to analyze the $n$-dimensional case, let us introduce the following notation. The convex hull of the points $M(t)$ is called the \emph{Tamari polytope}. The \emph{canopy} of the tree $t$ is an element of the set $\{\pm \}^{n}$ corresponding to the orientation of the interior leaves. If the leaf points to the left (resp.\ right), then we take $-$, resp. $+$. Of course we discard the two extremal leaves, whose orientation is fixed. We denote by
$$\psi : PBT_{n+2} \to \{\pm \}^{n}$$
this map. Among the trees with a given canopy, we single out the tree which is constructed as follows. We first draw the outer part of the tree.  Then for each occurence of $-$ we draw an edge which goes all the way to the right side of the tree (it is a left leaf). Then we complete the tree by drawing the right leaves. For instance:
$$\sigma(-) = \arbreBA, \quad \sigma(-,+) = \arbreCAB, \quad \sigma(-,-) = \arbreCBA .$$
This construction gives a section to $\psi$ that we denote by
$$\sigma : \{\pm \}^{n} \to PBT_{n+2} .$$

\begin{proposition} The Tamari polytope $KT^{n}$ is a hypercube shaped polytope, with extremal points $M(\sigma(\aa))$, for $\aa\in 
\{\pm \}^{n}$. For any tree $t$ the point $M(t)$ lies on a face of this hypercube containing $M(\sigma\psi(t))$.
\end{proposition}

\begin{proof} It is easily seen by induction that the convex hull of the points $M(\sigma(\aa))$ form a (combinatorial) hypercube.

Up to a change of orientation, this is the cubical version of the associahedron described in \cite{Loday02} section 2.5 (see also \cite{Loday11} Appendix 1). It is also described in \cite{SaneblidzeUmble04}.
\end{proof}

The Tamari polytope shares the following property with the standard permutohedron: all the edges have the same length.

Though Tamari himself does not consider this construction in his thesis, a close collaborator, Mrs de Foug\`eres, worked out some variations in \cite{Fougeres64}.

In 1963 Jim Stasheff \cite{Stasheff63} discovered independently the associahedron, first as a contractible cell complex, in his work on the structure of the loop spaces. It was later recognized to be realizable as a convex polytope, see for instance \cite{Stasheff97} Appendix B. In 2004 I gave in \cite{Loday04} an easy construction with integral coordinates as follows. It is usually described in terms of trees, but I will translate it in terms of parenthesized words. 

Given a parenthesized word of length $n$, for instance $((x_{0} x_{1})(x_{2}x_{3}))$, we associate to it a point in the euclidean space with coordinates computed as follows. The $i$th coordinate ($i$ ranging from $0$ to $n$) is the product of two numbers $a_{i}$ and $b_{i}$. We consider the smallest subword which contains both $x_{i}$ and $x_{i+1}$. Then $a_{i}$ is the number of opening parentheses standing to the left of $x_{i}$ and  $b_{i}$ is the number of closing parentheses standing to the right of $x_{i+1}$ in the subword. In the example at hand we get $1\ 4 \ 1$. It gives rise to the following polytopes in low dimension:

$$ \KdeuxA\hskip3cm \KtroisA\qquad$$

$$ \KKK^2\hskip5cm \KKK^3\qquad $$

Since then several interesting variations for the associahedron itself and for other families of polytopes have been given along the same lines, see for instance \cite{HohlwegLange07, Devadoss09, Forcey08a, Forcey08b, Postnikov09, PilaudSantos09}.

\section{Associahedron and the trefoil knot} Let us end this paper with a surprizing relationship which is not so-well-known. If we draw a path on the 3-dimensional associahedron from the center of each quadrangle to the center of the other quadrangles via the center of the pentagons, alternating over and under as we reenter a quadrangle, then we get the trefoil knot in Fig.~\ref{fig:trefoilknot}. 

\begin{figure}[h]
\begin{center}
\includegraphics[scale=0.4]{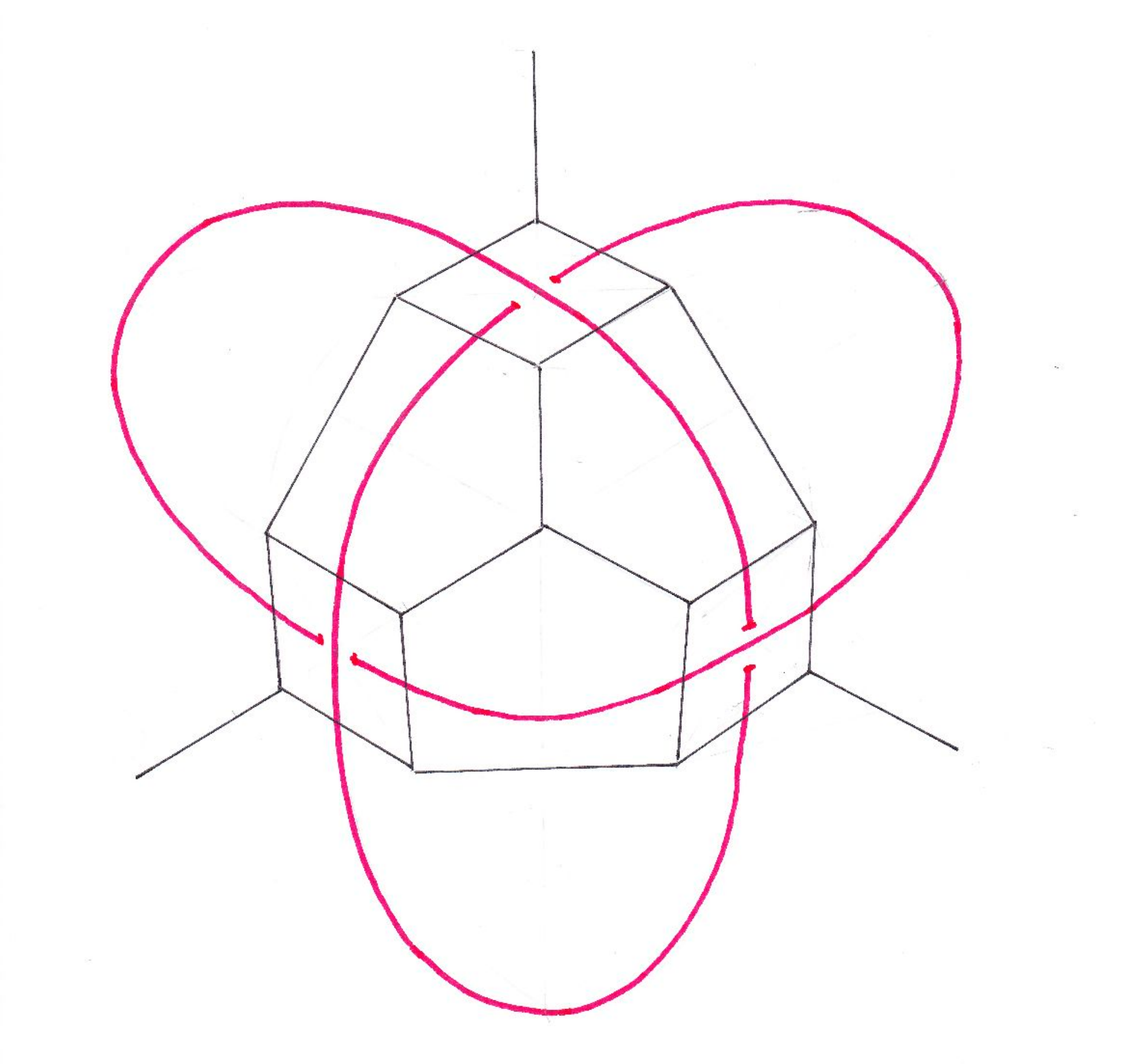}
\parbox{5.5cm}{
\caption{Trefoil knot \label{fig:trefoilknot}}
}
\end{center}
\end{figure}

The same process applied to the 3-dimensional cube gives rise to the Borromean rings. In the cube case we know how to relate the various invariants of this link: Philip Hall identity, triple Massey product. Nothing similar is known in the associahedron case so far.

\end{document}